\documentclass[a4paper,10pt]{article}
\usepackage[latin1]{inputenc}
\usepackage[T1]{fontenc}
\usepackage{amsmath}
\usepackage{amsfonts}
\usepackage{amstext}
\usepackage{amsthm}
\usepackage[all]{xy}
\usepackage{geometry}
\usepackage{srcltx}
\usepackage{graphics}
\usepackage{epsfig}
\usepackage{subfigure}
\usepackage{slashed}
\usepackage{color}
\usepackage{amssymb}
\usepackage{srcltx}
\newtheorem{theorem}{Theorem}[section]
\newtheorem{lemma}[theorem]{Lemma}

\newtheorem{rem}[theorem]{Remark}
\newtheorem{prop}[theorem]{Proposition}
\newtheorem{cor}[theorem]{Corollary}

\DeclareMathOperator{\sfl}{sf}

\DeclareMathOperator{\sgn}{sgn}

\title{Spectral flow and bifurcation for a class of strongly indefinite elliptic systems}
\author{Nils Waterstraat}

\begin{document}
\date{}
\maketitle


\begin{abstract}
\noindent We consider bifurcation of solutions from a given trivial branch for a class of strongly indefinite elliptic systems via the spectral flow. Our main results establish bifurcation invariants that can be obtained from the coefficients of the systems without using explicit solutions of their linearisations at the given branch. Our constructions are based on a comparison principle for the spectral flow and a generalisation of a bifurcation theorem due to Szulkin.  
\end{abstract}

\section{Introduction}
Let $\Omega$ be a domain in $\mathbb{R}^N$ for some $N\in\mathbb{N}$ which we assume to have a smooth boundary. Let $a,b,c:I\times\overline{\Omega}\rightarrow\mathbb{R}$ and $G:I\times\overline{\Omega}\times\mathbb{R}^2\rightarrow\mathbb{R}$ be $C^2$-functions, where $I:=[0,1]$ denotes -- here and throughout the paper -- the unit interval. We denote by $G'_u$ and $G'_v$ the partial derivatives of $G$ with respect to the components in $\mathbb{R}^2$, respectively, we assume that $G'_u(\lambda,x,0,0)=G'_v(\lambda,x,0,0)=0$ for all $(\lambda,x)\in I\times\Omega$ and we consider the systems of elliptic partial differential equations

\begin{equation}\label{equI}
\left\{
\begin{aligned}
-\Delta u&= b_\lambda(x)u+c_\lambda(x)v+G'_v(\lambda,x,u,v)&& \,\text{in}\,\,\Omega,\\
-\Delta v&= a_\lambda(x)u+b_\lambda(x)v+G'_u(\lambda,x,u,v)&& \,\text{in}\,\,\Omega,\\
u&=v=0&&\text{on}\,\,\partial\Omega
\end{aligned}
\right.
\end{equation}
depending on the parameter $\lambda\in I$. Clearly, under the mentioned assumptions the constant function $(u,v)\equiv0$ is a solution of \eqref{equI} for all values of $\lambda$ and the aim of this article is to investigate bifurcation from this trivial branch of solutions $I\times\{0\}$. Here, a bifurcation point of \eqref{equI} is an instant $\lambda^\ast\in I$ for which there is a sequence $\{(\lambda_n,u_n,v_n)\}_{n\in\mathbb{N}}$ such that $(u_n,v_n)\neq 0$ is a weak solution of \eqref{equI} for $\lambda_n$, $\lambda_n\rightarrow\lambda^\ast$ and $u_n,v_n\rightarrow0$ in the Sobolev space $H^1_0(\Omega,\mathbb{R})$ for $n\rightarrow\infty$. Our methods are based on a bifurcation theorem for critical points of families of functionals due to Fitzpatrick, Pejsachowicz and Recht \cite{SFLPejsachowicz}, which was recently improved by Pejsachowicz and the author in \cite{BifJac}. In order to explain this theorem briefly, let $f:I\times H\rightarrow\mathbb{R}$ be a family of $C^2$ functionals which are defined on a Hilbert space $H$ and such that $0\in H$ is a critical point of all $f_\lambda:=f(\lambda,\cdot):H\rightarrow\mathbb{R}$. A \textit{bifurcation point} of critical points for $f$ is an instant $\lambda^\ast\in I$ such that every neighbourhood of $(\lambda^\ast,0)$ in $I\times H$ contains elements $(\lambda,z)$ such that $z\neq 0$ is a critical point of $f_\lambda$. If we now represent the second derivatives  $D^2_0f_\lambda$ of $f$ at the critical point $0$ against the scalar product of $H$, then we obtain a path $L=\{L_\lambda\}_{\lambda\in I}$ of selfadjoint operators. It is readily seen from the Implicit Function Theorem that $L_{\lambda^\ast}$ is not invertible if $\lambda^\ast\in I$ is a bifurcation point of critical points of $f$.\\
The \textit{spectral flow} is an integer valued homotopy invariant for paths of selfadjoint Fredholm operators, which has been used in Global Analysis for about 40 years. We will say more about it in the next section, but here we just want to mention that a non-vanishing spectral flow implies that the path contains non-invertible operators. Its relevance for bifurcation of critical points for families of functionals was clarified in \cite{SFLPejsachowicz}: if the selfadjoint operators $L_\lambda$, which are induced by the Hessians of $f$ at $0$, are Fredholm, then a non-vanishing spectral flow is a sufficient condition for the existence of a bifurcation of critical points of $f$. Let us point out that if the operators $L_\lambda$ have finite Morse indices, then the spectral flow of $L$ is just the difference of the Morse indices of $L_0$ and $L_1$, and so the bifurcation theorem \cite{SFLPejsachowicz} is a classical assertion in variational bifurcation theory in this case. In contrast, it is often hard to compute the spectral flow of a given path of operators when the Morse indices are infinite (cf. e.g. \cite{AleSmaleIndef}).\\
The aim of this article is to show that for the indefinite elliptic systems \eqref{equI}, where the Morse indices of the corresponding operators $L_\lambda$ are indeed infinite, the spectral flow can be computed, or at least estimated, so that \cite{SFLPejsachowicz} can be used to derive bifurcation criteria. To our best knowledge, such easily computable bifurcation invariants that are induced by the spectral flow have not been obtained for partial differential equations before.\\
In the following section, we introduce a family of $C^2$-functionals $f:I\times E\rightarrow\mathbb{R}$ which is defined on the Sobolev space $E:=H^1_0(\Omega,\mathbb{R}^2)$ and which is such that the critical points of $f_\lambda:=f(\lambda,\cdot):E\rightarrow\mathbb{R}$ are precisely the weak solutions of \eqref{equI}. In particular, $0\in E$ is a critical point of each $f_\lambda$ and we can deduce the existence of a bifurcation from the zero-solution of \eqref{equI} by considering bifurcation of critical points from $0$ for the family of functionals $f$. We will state below conditions on the map $G$ which ensure that the Hessians $D^2_0f_\lambda$ of $f_\lambda$ at $0\in E$ exist, and that elements in the kernel of the representations $L_\lambda$ of $D^2_0f_\lambda$ on $E$ are the solutions of the linearised equation

\begin{equation}\label{equlinI}
\left\{
\begin{aligned}
-\Delta u&= b_\lambda(x)u+c_\lambda(x)v&& \,\text{in}\,\,\Omega,\\
-\Delta v&= a_\lambda(x)u+b_\lambda(x)v&& \,\text{in}\,\,\Omega,\\
u&=v=0&&\text{on}\,\,\partial\Omega.
\end{aligned}
\right.
\end{equation}
Since the operators $L_\lambda$ are readily seen to be Fredholm in this case, we can use the abstract bifurcation theorem \cite{SFLPejsachowicz}, and consequently, we will be concerned with the spectral flow of the corresponding path $L=\{L_\lambda\}_{\lambda\in I}$. In Theorem \ref{thm:comp}, which we consider as our main result of this paper, we estimate the spectral flow in terms of the coefficients of \eqref{equlinI} at $\lambda=0$ and $\lambda=1$, which is enough to conclude that it does not vanish and so implies the existence of a bifurcation of solutions for the nonlinear equations \eqref{equI}.\\
Another objective of this paper is to consider the special case in which the maps $a,b$ and $c$ do not depend on $x\in\Omega$, i.e.

\begin{equation}\label{equII}
\left\{
\begin{aligned}
-\Delta u&= b_\lambda u+c_\lambda v+G'_v(\lambda,x,u,v)&& \,\text{in}\,\,\Omega,\\
-\Delta v&= a_\lambda u+b_\lambda v+G'_u(\lambda,x,u,v)&& \,\text{in}\,\,\Omega,\\
u&=v=0&&\text{on}\,\,\partial\Omega.
\end{aligned}
\right.
\end{equation}
For these equations, we compute the spectral flow of the corresponding path of operators $L$ exactly in terms of an integral index that can be constructed from the coefficients of the linearised equations

\begin{equation}\label{equlinII}
\left\{
\begin{aligned}
-\Delta u&= b_\lambda u+c_\lambda v&& \,\text{in}\,\,\Omega,\\
-\Delta v&= a_\lambda u+b_\lambda v&& \,\text{in}\,\,\Omega,\\
u&=v=0&&\text{on}\,\,\partial\Omega
\end{aligned}
\right.
\end{equation}
for $\lambda=0$ and $\lambda=1$. The idea of this index goes back to Li and Liu \cite{LiLiu}, who used a similar construction in their study of existence of periodic solutions of asymptotically quadratic Hamiltonian systems. Their index was later applied in bifurcation theory for periodic solutions of Hamiltonian systems, e.g., by Szulkin in \cite{Szulkin}, and by Fitzpatrick, Pejsachowicz and Recht in \cite{SFLPejsachowiczII} who in particular used it to compute the spectral flow for autonomous Hamiltonian systems. It has some interest in its own that we compute in our Theorem \ref{mainthmII} the spectral flow for the equations \eqref{equlinII} by an index that is very much reminiscent of Li and Liu's index from \cite{LiLiu}. The adaption of Li and Liu's index for Hamiltonian systems to the elliptic systems \eqref{equlinII} follows closely Szulkin's work \cite[\S 5]{Szulkin} (cf. also \cite[\S 9]{Kryszewski}), who investigated the bifurcation problem for \eqref{equII} in the special case that $a,b,c$ and $G$ depend linearly on $\lambda$, i.e.

\begin{equation}\label{equSzulkin}
\left\{
\begin{aligned}
-\Delta u&= \lambda(b u+c v+G'_v(x,u,v))&& \,\text{in}\,\,\Omega,\\
-\Delta v&= \lambda(a u+b v+G'_u(x,u,v))&& \,\text{in}\,\,\Omega,\\
u&=v=0&&\text{on}\,\,\partial\Omega,
\end{aligned}
\right.
\end{equation}
by using infinite dimensional Morse theory for strongly indefinite functionals. We reobtain Szulkin's results in Corollary \ref{Corollary-Szulkin} as a consequence of our Theorem \ref{mainthmII}, and we will also assess our main Theorem \ref{thm:comp} for the equations \eqref{equSzulkin} below.\\
The paper is organised as follows. In the next section we introduce the family of functionals $f:I\times E\rightarrow\mathbb{R}$ which have as critical points the solutions of the nonlinear equations \eqref{equI}. Moreover, we recall the definition of the spectral flow, its main properties and its use in variational bifurcation theory. This leads us directly to a first theorem on bifurcation for \eqref{equI}. In the subsequent Section 3, we consider the equations \eqref{equII}. We construct an integral invariant for the coefficients of the systems \eqref{equlinII} and use our theorem from the previous section to show that the non-vanishing of this number causes bifurcation of the equations \eqref{equII}. In Section 4 we consider again the general systems \eqref{equI} and we use a comparison principle for the spectral flow to find criteria on the coefficients of \eqref{equlinI} to obtain bifurcation points for \eqref{equI}. This involves in particular the index that we introduced for the systems \eqref{equlinII} in the previous section. In the fifth section, we consider the case $N=1$, i.e. when the equations \eqref{equI} are ordinary differential equations. Since in this case we can estimate the dimension of the solution spaces of \eqref{equlinI} from above, we can use a result from \cite{BifJac} to obtain not only the existence of bifurcation points but also an estimate on their number. The paper ends with a short appendix in which we elaborate some folklore results on families of compact operators and projections.


\section{Spectral flow and bifurcation for \eqref{equI}}\label{sect:setting}
Let $\Omega\subset\mathbb{R}^N$ be a bounded domain with smooth boundary $\partial\Omega$. In what follows we assume that: 

\begin{itemize}
\item[(A1)] $a,b,c:I\times\overline{\Omega}\rightarrow\mathbb{R}$ and $G:I\times\overline{\Omega}\times\mathbb{R}^2\rightarrow\mathbb{R}$ are $C^2$-functions.
\item[(A2)] $G'_u$ and $G'_v$ are bounded and 

\[|G'_u(\lambda,x,u,v)|+|G'_v(\lambda,x,u,v)|=o(|u|+|v|)\]
as $|u|+|v|\rightarrow 0$ uniformly in $(\lambda,x)\in I\times\overline{\Omega}$.
\item[(A3)] $D^2G(\lambda,x,0,0)=0$ for all $(\lambda,x)\in I\times\Omega$, where $D^2G(\lambda,x,u,v)$ denotes the Hessian matrix of $G(\lambda,x,\cdot,\cdot):\mathbb{R}^2\rightarrow\mathbb{R}$ at $(u,v)\in\mathbb{R}^2$. 
\end{itemize}
Moreover, if $N>1$, we shall also assume that

\begin{itemize}
\item[(A4)] there exists $C\geq 0$ such that

\begin{align*}
\|D^2G(\lambda,x,u,v)\|\leq C(1+|u|+|v|)^{p-1},\quad (\lambda,x)\in I\times\overline{\Omega},\,\,u,v\in\mathbb{R},
\end{align*}
where $1\leq p<\frac{N+2}{N-2}$ if $N>2$ and $1\leq p<\infty$ if $N=2$. 
\end{itemize}
Note that the constant function $(u,v)\equiv 0$ is a solution of \eqref{equI} for all $\lambda\in I$ by (A2), and the aim of this article is to study bifurcation of (weak) solutions of \eqref{equI} from this trivial branch.\\
Let now $H^1_0(\Omega,\mathbb{R})$ be the usual Sobolev space with scalar product 

\[\langle u_1,u_2\rangle_{H^1_0(\Omega,\mathbb{R})}=\int_\Omega{\langle\nabla u_1,\nabla u_2\rangle\,dx}\]
and we set $E:=H^1_0(\Omega,\mathbb{R})\times H^1_0(\Omega,\mathbb{R})$ which is a Hilbert space with respect to 

\[\langle (u_1,v_1),(u_2,v_2)\rangle_E=\langle u_1,u_2\rangle_{H^1_0(\Omega,\mathbb{R})}+\langle v_1,v_2\rangle_{H^1_0(\Omega,\mathbb{R})}.\]
We consider the map $f:I\times E\rightarrow\mathbb{R}$ given by

\begin{align}\label{functional}
f_\lambda(z)=\int_\Omega{\langle\nabla u,\nabla v\rangle\,dx}-\frac{1}{2}\int_\Omega{a_\lambda(x)u^2+2b_\lambda(x)uv+c_\lambda(x)v^2\,dx}-\int_\Omega{G(\lambda,x,u,v)\,dx},
\end{align}
where $z=(u,v)\in E$, and we note that $f$ is $C^2$ under our assumptions (A1), (A2) and (A4) (cf. \cite{Kryszewski}). The critical points of $f_\lambda$ are precisely the weak solutions of equation \eqref{equI}, and in particular $0\in E$ is a critical point of all functionals $f_\lambda$. We say that $\lambda^\ast\in I$ is a \textit{bifurcation point} of weak solutions for the equations \eqref{equI}, if every neighbourhood of $(\lambda^\ast,0)\in I\times E$ contains some $(\lambda,z)\neq (\lambda,0)$, where $z$ is a weak solution of \eqref{equI}; or equivalently, a critical point of $f_\lambda$. Consequently, in order to investigate bifurcation of \eqref{equI} from the trivial branch of solutions we need to study bifurcation of critical points of \eqref{functional} from the branch $I\times\{0\}\subset I\times E$. For this we consider the Hessians of $f_\lambda$ at $0\in E$, which are given by

\begin{align}\label{Hessian}
\begin{split}
D^2_0f_\lambda(z,\overline{z})&=\int_\Omega{\langle\nabla u,\nabla \overline{v}\rangle\,dx}+\int_\Omega{\langle\nabla \overline{u},\nabla v\rangle\,dx}\\
&-\int_\Omega{a_\lambda(x)u\overline{u}+b_\lambda(x)(\overline{u}v+u\overline{v})+c_\lambda(x)v\overline{v}\,dx} \quad z=(u,v),\,\,\overline{z}=(\overline{u},\overline{v}),
\end{split}
\end{align} 
where we use Assumption (A3). Let us denote by $L_\lambda$ the Riesz representations of $D^2_0f_\lambda$, i.e. the bounded selfadjoint operators on $E$ defined by

\begin{align}\label{Riesz}
\langle L_\lambda z,\overline{z}\rangle_E=D^2_0f(z,\overline{z}),\quad z,\overline{z}\in E.
\end{align} 
Then $L_\lambda=T+K_\lambda$, where $T:E\rightarrow E$ is the selfadjoint invertible operator given by 

\begin{align}\label{T}
Tz=T(u,v)=(v,u),\quad z=(u,v)\in E.
\end{align}
Moreover, the operator $K_\lambda$, which is uniquely determined by

\begin{align}\label{compactness}
\langle K_\lambda z,\overline{z}\rangle_E=-\int_\Omega{a_\lambda(x)u\overline{u}+b_\lambda(x)(\overline{u}v+u\overline{v})+c_\lambda(x)v\overline{v}\,dx},\quad z=(u,v), \overline{z}=(\overline{u},\overline{v}),
\end{align}
is compact since the right hand side in \eqref{compactness} extends to a bounded quadratic form on $L^2(\Omega,\mathbb{R}^2)$ and $E$ is compactly embedded in this space (cf. e.g. \cite[Lemma 3.1]{CalcVar}). Consequently, $L=\{L_\lambda\}_{\lambda\in I}$ is a path of selfadjoint Fredholm operators to which we can assign the \textit{spectral flow}.\\
Let now $H$ be an arbitrary separable real Hilbert space. The spectral flow is an integer-valued index for paths $L=\{L_\lambda\}_{\lambda\in I}$ of selfadjoint Fredholm operators $L_\lambda$ on $H$ which we denote by $\sfl(L,I)$. It was introduced by Atiyah, Patodi and Singer in the Seventies in \cite{AtiyahPatodi} and since then it has reappeared in many different areas of geometry and analysis (we refer to \cite{BifJac} for a detailed list of references). Here we introduce it along the lines of \cite{SFLPejsachowicz}, and discuss an application to bifurcation of critical points of families of functionals from the same reference and \cite{BifJac}. In what follows, we call a path of selfadjoint Fredholm operators \textit{admissible} if its endpoints are invertible. Moreover, we denote by $\Phi_S(H)$ the space of all bounded selfadjoint Fredholm operators equipped with the norm topology. In order to shorten the presentation, we use an axiomatic description from \cite{SFLUniqueness}. Accordingly, the spectral flow is the unique map which assigns to each admissible path $L=\{L_\lambda\}_{\lambda\in I}$ in $\Phi_S(H)$ an integer such that:   

\begin{itemize}
	\item(Normalisation) If $L_\lambda$ is invertible for all $\lambda\in I$, then 
	
	\[\sfl(L,I)=0.\]
	\item(Additivity) If $H=H_1\oplus H_2$ and $L_\lambda(H_i)\subset H_i$ for all $\lambda\in I$ and $i=1,2$, then 
	\[\sfl(L,I)=\sfl(L\mid_{H_1},I)+\sfl(L\mid_{H_2},I).\]
	\item(Homotopy) If $\{h_{(\lambda,s)}\}_{(\lambda,s)\in I\times I}$ is a family in $\Phi_S(H)$ such that $h(0,s)$ and $h(1,s)$ are invertible for all $s\in I$, then 
	
	\[\sfl(h(\cdot,0),I)=\sfl(h(\cdot,1),I).\]
	\item(Dimension) If $\dim H<\infty$ then
	\[\sfl(L,I)=\mu_{Morse}(L_0)-\mu_{Morse}(L_1),\]
	where $\mu_{Morse}$ denotes the Morse index, i.e. the number of negative eigenvalues counted with multiplicities.    
\end{itemize}
Clearly, by reparameterising, the spectral flow can also be defined for paths which are parametrised by a general compact interval $[\lambda_0,\lambda_1]$. If $L=\{L_\lambda\}_{\lambda\in [\lambda_0,\lambda_1]}$ is an admissible path of selfadjoint Fredholm operators, then we denote its spectral flow by $\sfl(L,[\lambda_0,\lambda_1])$, and we note the following property for later reference:

\begin{itemize}
\item(Concatenation) If $\lambda_0<\lambda_1<\lambda_2$ and $L_{\lambda_1}$ is invertible, then 

\[\sfl(L,[\lambda_0,\lambda_2])=\sfl(L,[\lambda_0,\lambda_1])+\sfl(L,[\lambda_1,\lambda_2]).\]
\end{itemize}
Let us now consider continuous maps $f:I\times H\rightarrow\mathbb{R}$ of $C^2$ functionals $f_\lambda:=f(\lambda,\cdot):H\rightarrow\mathbb{R}$ such that the derivatives $Df_\lambda$ and $D^2f_\lambda$ depend continuously on $\lambda$, and let us assume that $D_0f_\lambda=0$, i.e. $0\in H$ is a critical point of $f_\lambda$ for all $\lambda\in I$. Recall that a \textit{bifurcation point} of critical points of $f$ is an instant $\lambda^\ast\in I$ such that every neighbourhood of $(\lambda^\ast,0)$ in $I\times H$ contains elements $(\lambda,u)$ where $u\neq 0$ is a critical point of $f_\lambda$. The main theorems in \cite{SFLPejsachowicz} and \cite{BifJac} state that if the Riesz representations $L_\lambda$ of $D^2_0f_\lambda$ are Fredholm for all $\lambda$ and $L_0,L_1$ are invertible, then there is a bifurcation of critical points of $f$ if $\sfl(L,I)\neq 0$.\\ 
Let us now come back to the differential equations \eqref{equI} and the functionals \eqref{functional} on the Hilbert space $E$ for which the corresponding operators $L_\lambda$ are the Riesz representations of \eqref{Hessian}. By standard regularity theory, it follows that the kernels of $L_\lambda$ consist of solutions of the linearised equations \eqref{equlinI}. Since $L_\lambda$ is Fredholm and selfadjoint, its Fredholm index vanishes, and so we conclude that $L_\lambda$ is non-invertible if and only if \eqref{equlinI} has a non-trivial solution. Let us mention in passing that it is readily seen from the implicit function theorem that $L_{\lambda^\ast}$ is not invertible if $\lambda^\ast$ is a bifurcation point, which provides information about the location of possible bifurcation points. Finally, we can summarise the previous discussion as follows.

\begin{theorem}\label{mainthmI}
Let $\Omega\subset\mathbb{R}^N$ be a bounded domain having a smooth boundary and let the functions $a,b,c$ and $G$ in \eqref{equI} satisfy (A1)-(A4). If the linear systems \eqref{equlinI} have no non-trivial solution for $\lambda=0,1$ and $\sfl(L,I)\neq 0$, then there is a bifurcation point $\lambda^\ast\in(0,1)$ for the family of equations \eqref{equI}.
\end{theorem}
\noindent
Of course, the difficult point when applying Theorem \ref{mainthmI} to the equations \eqref{equI} is to compute $\sfl(L,I)$, or at least to find conditions that ensure its non-triviality. In the remainder of this article we will be concerned with this problem. At first, we want to review a method for computing spectral flows that has been applied several times in the past in other settings (e.g. for Hamiltonian systems in \cite{SFLPejsachowiczII} and for partial differential equations in \cite{AleIchDomain}, \cite{AleBall}, \cite{ProcHan} and \cite{AleSmaleIndef}).\\
Let us assume for the remainder of this section that the path $\{L_\lambda\}_{\lambda\in I}$ is $C^1$ in $\mathcal{L}(E)$. We call an instant $\lambda_0\in I$ a crossing if $L_{\lambda_0}$ is non-invertible, which is, as we have already observed, the case if and only if \eqref{equlinI} has a non-trivial solution. Given a crossing $\lambda_0$, we obtain a quadratic form on $\ker L_{\lambda_0}$ by

\[\Gamma(L,\lambda_0):\ker L_{\lambda_0}\rightarrow\mathbb{R},\quad \Gamma(L,\lambda_0)[u]=\langle\frac{d}{d\lambda}\mid_{\lambda=\lambda_0} L_\lambda u,u\rangle_H,\]
and we say that a crossing is \textit{regular} if $\Gamma(L,\lambda_0)$ is non-degenerate. Let now $\lambda_0$ be a regular crossing of $L$. One can show that regular crossings are isolated and hence there is $\varepsilon>0$ such that $L_\lambda$ is invertible for all $\lambda$ in the punctured neighbourhood $[\lambda_0-\varepsilon,\lambda_0+\varepsilon]\setminus\{\lambda_0\}$. We obtain from the previously mentioned bifurcation theorem \cite{SFLPejsachowicz} that there is a bifurcation point for $f$ in $[\lambda_0-\varepsilon,\lambda_0+\varepsilon]$ if $\sfl(L,[\lambda_0-\varepsilon,\lambda_0+\varepsilon])\neq 0$. As $L_\lambda$ is not invertible at bifurcation points by the implicit function theorem, it follows in this case that the obtained bifurcation point is $\lambda_0$. By a theorem due to Robbin and Salamon \cite{Robbin-Salamon} (cf. also \cite{SFLPejsachowicz} and \cite{Homoclinics}), $\sfl(L,[\lambda_0-\varepsilon,\lambda_0+\varepsilon])$ is given by the signature of the quadratic form $\Gamma(L,\lambda_0)$ on $\ker L_{\lambda_0}$. From \eqref{Hessian} and \eqref{Riesz} we obtain that

\[\Gamma(L,\lambda_0)[z]=-\int_\Omega{\dot a_{\lambda_0}(x)u^2+2\dot b_{\lambda_0}(x)uv+\dot c_{\lambda_0}(x)v^2\,dx},\quad z=(u,v)\in\ker L_{\lambda_0},\]    
where $\dot{}$ denotes the derivative with respect to the parameter $\lambda$. If we use that a quadratic form is non-degenerate and of non-vanishing signature if it is positive or negative definite, we obtain from Sylvester's criterion the following result.

\begin{theorem}
Let $\Omega\subset\mathbb{R}^N$ be a bounded domain having a smooth boundary and let the functions $a,b,c$ and $G$ in \eqref{equI} satisfy (A1)-(A4). If the linear systems \eqref{equlinI} have a non-trivial solution for $\lambda=\lambda_0\in(0,1)$, $\dot a_{\lambda_0}(x)\neq 0$ for all $x\in\Omega$ and 

\begin{align}\label{posdef}
	\dot a_{\lambda_0}(x)\dot c_{\lambda_0}(x)-\dot b^2_{\lambda_0}(x)>0,\quad x\in\Omega,
	\end{align}
	 then $\lambda_0$ is a bifurcation point for \eqref{equI}.
\end{theorem}
\noindent
Equation \eqref{posdef} is a convenient criterion for the existence of bifurcation points, however, we want to point out a drawback of this approach: the non-triviality of $\ker L_{\lambda_0}$ and so the existence of non-trivial solutions of \eqref{equlinI} needs to be known. The aim of the following sections is to present approaches to the bifurcation problem of \eqref{equI} which only uses information about the coefficients of \eqref{equlinI} and not about possible solutions for parameter values $\lambda\in(0,1)$.


\section{Index and bifurcation for \eqref{equII}}\label{sect:index}
In this section we consider the equations \eqref{equII}, where we again assume throughout (A1)-(A4). Our first aim is to construct an invariant for the equations \eqref{equlinII}, which we will use below to compute the spectral flow of the associated path $L=\{L_\lambda\}_{\lambda\in I}$ introduced in \eqref{Riesz} in order to obtain the existence of bifurcation from Theorem \ref{mainthmI}. The following construction is based on Li and Liu's work \cite{LiLiu} for Hamiltonian systems, which was adapted to the equations \eqref{equSzulkin} by Szulkin in \cite{Szulkin}.\\
Let $\{e_k\}_{k\in\mathbb{N}}$ be an orthonormal basis of $H^1_0(\Omega,\mathbb{R})$ such that $-\Delta e_k=\lambda_k e_k$, and let us recall that the eigenvalues $\lambda_k$ are all positive and $\lambda_k\rightarrow\infty$ for $k\rightarrow\infty$. Now $\{\frac{1}{\sqrt{2}}(e_k,-e_k),\frac{1}{\sqrt{2}} (e_k,e_k)\}_{k\in\mathbb{N}}$ is an orthonormal basis of $E$ and we get an orthonormal decomposition 

\[E=H^1_0(\Omega,\mathbb{R})\oplus H^1_0(\Omega,\mathbb{R})=\bigoplus_{k\in\mathbb{N}} E_k,\]
where $E_k$ is the two dimensional space generated by $(e_k,-e_k)$ and $(e_k,e_k)$. Since $T(u,v)=(v,u)$ for all $(u,v)\in E$ (cf. \eqref{T}), we see that $T(E_k)\subset E_k$. By the following lemma, also the operators $K_\lambda$ in \eqref{compactness} leave the spaces $E_k$ invariant.

\begin{lemma}\label{invariant}
Let $P_k$ and $P_l$ denote the orthogonal projections in $E$ onto $E_k$ and $E_l$, respectively. If $k\neq l$, then

\[P_kK_\lambda P_l=0,\quad \lambda\in I.\]
\end{lemma}
  
\begin{proof}
If $z,\overline{z}\in E$, then $P_kz$ and $P_l\overline{z}$ are linear combinations of $(e_k,e_k), (-e_k,e_k)$ and $(e_l,e_l), (-e_l,e_l)$, respectively. Since the coefficients $a,b,c$ do not depend on $x\in\Omega$, it follows from \eqref{compactness} that 

\[\langle P_kK_\lambda P_lz,\overline{z}\rangle_E=\langle K_\lambda P_lz,P_k\overline{z}\rangle_E=m\int_{\Omega}{e_ke_l\,dx}\]
for some number $m\in\mathbb{R}$. However,
\[\int_\Omega{e_ke_l\,dx}=-\frac{1}{\lambda_k}\int_\Omega{(\Delta e_k)e_l\,dx}=\frac{1}{\lambda_k}\langle e_k,e_l\rangle_{H^1_0(\Omega,\mathbb{R})}=0\]
by Green's formula and so $P_kK_\lambda P_l=0$ for $k\neq l$.
\end{proof}
\noindent
We now define
\[L^k_\lambda:=P_kL_\lambda P_k=P_k(T+K_\lambda)P_k\mid_{E_k}=T+K_\lambda\mid_{E_k}:E_k\rightarrow E_k,\quad k\in\mathbb{N},\]
where $P_k:E\rightarrow E$ denotes the orthogonal projection onto $E_k$. If we set 

\[z=(u,v)=\frac{\alpha}{\sqrt{2}}(e_k,-e_k)+\frac{\beta}{\sqrt{2}}(e_k,e_k),\quad \alpha,\beta\in\mathbb{R}\]
then

\begin{align*}
D^2_0f_\lambda(z,z)=(\beta^2-\alpha^2)-\frac{1}{2\lambda_k}((a_\lambda-2b_\lambda+c_\lambda)\alpha^2+2(a_\lambda-c_\lambda)\alpha\beta+(a_\lambda+2b_\lambda+c_\lambda)\beta^2),
\end{align*} 
where we use that $\int_\Omega{e^2_k\,dx}=\frac{1}{\lambda_k}$, $k\in\mathbb{N}$. We obtain

\begin{align}\label{Lmatrix}
L^k_\lambda=\begin{pmatrix}
-1&0\\
0&1
\end{pmatrix}
-\frac{1}{2\lambda_k}\begin{pmatrix}
a_\lambda-2b_\lambda+c_\lambda& a_\lambda-c_\lambda\\
a_\lambda-c_\lambda& a_\lambda+2b_\lambda+c_\lambda
\end{pmatrix},\quad\lambda\in I,
\end{align}
with respect to the orthonormal basis $\{\frac{1}{\sqrt{2}}(e_k,-e_k),\frac{1}{\sqrt{2}}(e_k,e_k)\}$ of $E_k$. In particular, since $\lambda_k\rightarrow\infty$ as $k\rightarrow\infty$, there exists $k_0\in\mathbb{N}$ such that $L^k_\lambda$ is an isomorphism and

\[\sgn(L^k_\lambda)=\mu_{Morse}(-L^k_\lambda)-\mu_{Morse}(L^k_\lambda)=0,\quad\text{for all}\,\, k\geq k_0\,\, \text{and all}\,\,\lambda\in I.\]
Hence we can define for all $\lambda\in I$ an \textit{index} of the coefficient matrix

\[A_\lambda:=\begin{pmatrix}
a_\lambda&b_\lambda\\
b_\lambda&c_\lambda
\end{pmatrix}\]
of \eqref{equlinII} by

\[i(A_\lambda)=\frac{1}{2}\sum^{\infty}_{k=1}{\sgn(L^k_\lambda)}.\]
Note that if $L_\lambda$ is invertible, then $L^k_\lambda$ is invertible for all $k\in\mathbb{N}$ and so $\sgn(L^k_\lambda)$ is either $-2$, $0$ or $2$. Hence $i(A_\lambda)$ is an integer if $L_\lambda$ is invertible, whereas it is only a half-integer in general. The main Theorem of this section reads as follows:

\begin{theorem}\label{mainthmII}
Let $\Omega\subset\mathbb{R}^N$ be a bounded domain having a smooth boundary and let us assume that (A1)-(A4) hold. If \eqref{equlinII} has only the trivial solution for $\lambda=0,1$ and 

\[i(A_0)\neq i(A_1),\]
then there exists a bifurcation point for \eqref{equII} in $(0,1)$.
\end{theorem}

\begin{proof}
Let us recall that $L_\lambda$ is of the form $L_\lambda=T+K_\lambda$, where the operators on the right hand side were introduced in \eqref{T} and \eqref{compactness}, respectively. Moreover, $L_0$ and $L_1$ are invertible since \eqref{equlinII} has no non-trivial solutions for these parameter values by assumption.\\
We denote by $Q_n:=\sum^n_{k=1}{P_k}$ the orthogonal projection onto $\bigoplus^n_{k=1}{E_k}$ and by $Q^\perp_n$ the corresponding complementary projection, i.e. $Q^\perp_n=I_E-Q_n$. We note that 

\begin{align}\label{L}
\begin{split}
L_\lambda&=T+Q_nK_\lambda Q_n+Q^\perp_nK_\lambda Q_n+Q_nK_\lambda Q^\perp_n+Q^\perp_nK_\lambda Q^\perp_n\\
&=T+Q_nK_\lambda Q_n+Q^\perp_nK_\lambda Q^\perp_n,\quad n\in\mathbb{N},
\end{split}
\end{align}
where the second equality is a simple consequence of Lemma \ref{invariant}. We now claim that there are $n_0\in\mathbb{N}$ and a constant $C>0$ such that for all $n\geq n_0$

\begin{align}\label{invI}
\|Tu+Q_nK_\lambda Q_nu\|\geq 2C\,\|u\|,\quad u\in E,\,\, \lambda=0,1,
\end{align}
and

\begin{align}\label{invII}
\|Q^\perp_nK_\lambda Q^\perp_n\|\leq C,\quad \lambda\in [0,1].
\end{align}
The reader can find in Appendix \ref{CompactStrongConv} a proof of the fact that if $\{S_n\}_{n\in\mathbb{N}}$ is a sequence in $\mathcal{L}(E)$ which converges strongly to some $S\in\mathcal{L}(H)$, and $K_\lambda$ , $\lambda\in[0,1]$, is a continuous family of compact operators, then $S_nK_\lambda$ converges in norm to $SK_\lambda$ as $n\rightarrow\infty$, and the convergence is even uniform in $\lambda$. Consequently, since $Q^\perp_n$ converges strongly to $0$ as $n\rightarrow\infty$ and $\|Q^\perp_n\|=1$, we infer that  

\begin{align}\label{conv}
\|Q^\perp_nK_\lambda Q^\perp_n\|\leq \|Q^\perp_nK_\lambda\|\rightarrow 0,\quad n\rightarrow\infty.
\end{align}
Since $L_\lambda$ is invertible for $\lambda=0,1$, there is $C>0$ such that $\|L_\lambda u\|\geq 3C\,\|u\|$ for $u\in E$ and $\lambda=0,1$. We obtain from \eqref{L} and \eqref{conv} that there is $n_0$ such that $\|Tu+Q_nK_\lambda Q_nu\|\geq 2C\,\|u\|$ for all $n\geq n_0$, $u\in E$ and $\lambda=0,1$, which is \eqref{invI}. After possibly increasing $n_0$, we can assume that \eqref{invII} holds for the same constant $C>0$, where we use that the convergence in \eqref{conv} is uniform in $\lambda$.\\ 
We now assume that $n_0$ in \eqref{invI} and \eqref{invII} is sufficiently large such that $\sgn(L^k_\lambda)=0$ for all $\lambda\in I$ and all $k\geq n_0$, and we consider for some $n\geq n_0$ the homotopy $h:[0,1]\times[0,1]\rightarrow\Phi_S(E)$ defined by

\[h(t,\lambda)=T+Q_{n}K_\lambda Q_{n}+t\,Q^\perp_{n}K_\lambda Q^\perp_{n}.\]
By \eqref{invI} and \eqref{invII}, we conclude that

\[\|h(t,\lambda)u\|\geq C\,\|u\|,\quad u\in E,\,\,\lambda=0,1,\]
and hence $h(t,0)$ and $h(t,1)$ are invertible for all $t\in[0,1]$ since they are Fredholm of index $0$. The Homotopy Invariance Property of the spectral flow yields

\begin{align}\label{sflequI}
\sfl(L,I)=\sfl(\{T+Q_{n}K_\lambda Q_{n}\}_{\lambda\in I},I).
\end{align}
By Lemma \ref{invariant} we have

\begin{align*}
Q_{n}K_\lambda Q_{n}=\sum^{n}_{k,l=1}{P_kK_\lambda P_l}=\sum^{n}_{k=1}{P_kK_\lambda P_k},
\end{align*}
and since $T$ is also reduced by the projections $P_k$ it follows likewise that

\[Q_{n}T Q_{n}=\sum^{n}_{k,l=1}{P_kT P_l}=\sum^{n}_{k=1}{P_kT P_k}.\]
We obtain

\[T+Q_{n}K_\lambda Q_{n}=Q_{n}TQ_{n}+Q_{n}K_\lambda Q_{n}+Q^\perp_{n}T Q^\perp_{n}=\sum^{n}_{k=1}{(P_k(T+K_\lambda)P_k)}+Q^\perp_{n}T Q^\perp_{n}.\] 
Now the Additivity and the Normalisation Property of the spectral flow yield 

\begin{align*}
\begin{split}
\sfl(\{T+Q_{n} K_\lambda Q_{n}\}_{\lambda\in I},I)&=\sfl(\{\sum^{n}_{k=1}{(P_k(T+K_\lambda)P_k)}+Q^\perp_{n}T Q^\perp_{n}\}_{\lambda\in I},I)\\
&=\sum^{n}_{k=1}{\sfl(\{P_k(T+K_\lambda)P_k\}_{\lambda\in I},I)}=\sum^{n}_{k=1}{\sfl(L^k,I)},
\end{split}
\end{align*}
where we use that $Q^\perp_{n}T Q^\perp_{n}$ is an invertible operator on the image of $Q^\perp_n$. By the Dimension Property of the spectral flow, we obtain 

\begin{align}\label{sflMorse}
\sum^{n}_{k=1}{\sfl(L^k,I)}=\sum^{n}_{k=1}{(\mu_{Morse}(L^k_0)-\mu_{Morse}(L^k_1))}.
\end{align}
As $L_0$ and $L_1$ are invertible by assumption, we see that $L^k_0$ and $L^k_1$ are invertible for all $k\in\mathbb{N}$. Since the signature and the Morse index of an invertible symmetric $2\times 2$-matrix $B$ are related by $\frac{1}{2}\sgn B=1-\mu_{Morse}(B)$, we can rewrite the right hand side in \eqref{sflMorse} by

\begin{align*}
\sum^{n}_{k=1}{(\mu_{Morse}(L^k_0)-\mu_{Morse}(L^k_1))}&=\sum^{n}_{k=1}{\frac{1}{2}(\sgn(L^k_1)-\sgn(L^k_0))}=\frac{1}{2}\sum^{n}_{k=1}{\sgn(L^k_1)}-\frac{1}{2}\sum^{n}_{k=1}{\sgn(L^k_0)}\\
&=i(A_1)-i(A_0),
\end{align*}  
where we have used in the last step that $\sgn(L^k_\lambda)=0$ for all $\lambda\in I$ and all $k\geq n_0$ by our choice of $n_0$. Consequently, we have shown that

\begin{align}\label{sflformula}
\sfl(L,I)=i(A_1)-i(A_0),
\end{align} 
and now the assertion follows from Theorem \ref{mainthmI}.
\end{proof}

\begin{rem}
Let us point out that we have derived in the proof of Theorem \ref{mainthmII} in \eqref{sflformula} a spectral flow formula for the path $\{L_\lambda\}_{\lambda\in I}$, which is of independent interest. The spectral flow can also be defined for paths of unbounded selfadjoint Fredholm operators (cf. e.g. \cite{UnbSpecFlow} or \cite{Homoclinics}). Let us consider on $L^2(\Omega,\mathbb{R}^2)$ the differential operators $\mathcal{A}_\lambda$ on the domain $W=H^2(\Omega,\mathbb{R}^2)\cap H^1_0(\Omega,\mathbb{R}^2)$ defined by

\[\mathcal{A}_\lambda \begin{pmatrix} u\\v\end{pmatrix}:=\begin{pmatrix} -\Delta v\\-\Delta u\end{pmatrix}+\begin{pmatrix} a_\lambda&b_\lambda\\b_\lambda &c_\lambda\end{pmatrix}\begin{pmatrix} u\\v\end{pmatrix}.\]
Note that elements in the kernel of $\mathcal{A}_\lambda$ are the solutions of the equations \eqref{equlinII}. It can be shown that the spectral flow of the path $\mathcal{A}=\{\mathcal{A}_\lambda\}_{\lambda\in I}$ coincides with the spectral flow of the corresponding path $L=\{L_\lambda\}_{\lambda\in I}$ in \eqref{Riesz} (cf. \cite[Thm. 2.6]{CalcVar}), and so \eqref{sflformula} yields also a spectral flow formula for the differential operators $\mathcal{A}_\lambda$.    
\end{rem}
\noindent
As announced in the previous section, Theorem \ref{mainthmII} uses only the coefficients of the equation \eqref{equlinI} and no information about solutions of the linearisations \eqref{equlinII} for $\lambda\in(0,1)$.\\
Let us now consider the equations \eqref{equSzulkin} where $A_\lambda=\lambda\,A$ depends linearly on the parameter $\lambda$. Here we want to change the setting slightly and instead of restricting $\lambda$ to the unit interval $I$, we want to consider the case that $\lambda\in\mathbb{R}$. As before, we have for each $\lambda\in\mathbb{R}$ the integral number $i(A_\lambda)$. We obtain from Theorem \ref{mainthmII} the following result that was proved by Szulkin in \cite[\S 5]{Szulkin}.

\begin{cor}\label{Corollary-Szulkin}
Let $\Omega\subset\mathbb{R}^N$ be a bounded domain having a smooth boundary and let us assume that (A1)-(A4) hold, where $a_\lambda(x)=\lambda\,a$, $b_\lambda(x)=\lambda\,b$ and $c_\lambda(x)=\lambda\,c$ for some real numbers $a,b,c$ and $\lambda\in\mathbb{R}$. If $i(\lambda\,A)$ jumps at some $\lambda^\ast\in\mathbb{R}$, then $\lambda^\ast$ is a bifurcation point.
\end{cor}

\begin{proof}
We first note that the operators $L_\lambda$ are of the form $T+\lambda K$, where $T$ is invertible and $K$ is compact and does not depend on $\lambda$. Hence, by the spectral theory of compact operators, the set of all $\lambda\in\mathbb{R}$ for which $L_\lambda$ is not invertible is discrete. Secondly, if $L^k_\lambda$ is non-invertible for some $k\in\mathbb{N}$, then $L_\lambda$ is non-invertible as well.\\
Let us now assume that $i(\lambda\,A)$ jumps at some $\lambda^\ast$. Then there is $k\in\mathbb{N}$ such that $L^k_{\lambda^\ast}$ is not invertible. Hence $L_{\lambda^\ast}$ is not invertible and there is $\varepsilon>0$ such that $L_\lambda$ is invertible if $\lambda\in(\lambda^\ast-2\varepsilon,\lambda^\ast+2\varepsilon)\setminus\{\lambda_0\}$. Consequently, $L_{\lambda^\ast-\varepsilon}$ and $L_{\lambda^\ast+\varepsilon}$ are invertible and since $i(A_{\lambda^\ast-\varepsilon})\neq i(A_{\lambda^\ast+\varepsilon})$ the assertion follows from Theorem \ref{mainthmII}.
\end{proof}


\section{Bifurcation by Comparison}
For the considerations of this section, we want to introduce at first a theorem about the spectral flow that was proved in \cite{BifJac}. Before, we need to extend the definition of the spectral flow, which we recalled in the second section, to paths $L=\{L_\lambda\}_{\lambda\in I}$ in $\Phi_S(H)$ that do not have invertible endpoints, i.e. that are not admissible. Since $0$ is an isolated eigenvalue of finite multiplicity (cf. e.g. \cite[Lemma 2.2]{ProcHan}), there exists $\delta\geq0$ such that $L_0+\mu I_H$ and $L_1+\mu I_H$ are invertible for all $0<\mu\leq\delta$, where $I_H$ denotes the identity operator on $H$. We set

\[\sfl(L,I):=\sfl(L+\delta I_H,I).\]
Of course, if $L$ is admissible, then this definition coincides with the previous one by the Homotopy Invariance Property. In what follows, we write $T\leq S$ for $T,S\in\Phi_S(H)$ if

\[\langle Tz,z\rangle_E\leq\langle Sz,z\rangle_H,\quad z\in H.\] 
A proof of the following proposition can be found in \cite[\S 7]{BifJac}. 

\begin{prop}\label{comparison}
Let $L=\{L_\lambda\}_{\lambda\in I}$ and $M=\{M_\lambda\}_{\lambda\in I}$ be paths in $\Phi_S(H)$ such that $L_\lambda-M_\lambda$ is compact for all $\lambda\in I$. If 

\[L_0\leq M_0\,\,\text{and}\,\, M_1\leq L_1,\]
then 

\[\sfl(M,I)\leq\sfl(L,I).\]
\end{prop}
\noindent
Let us now consider again the systems \eqref{equI}. Note that the operators $K_\lambda$ in \eqref{compactness} can be written as 

\[\langle K_\lambda z,\overline{z}\rangle_E=\int_{\Omega}{\langle A_\lambda(x)z,\overline{z}\rangle\,dx},\]
where 

\[A_\lambda(x):=-\begin{pmatrix} a_\lambda(x)&b_\lambda(x)\\ b_\lambda(x)&c_\lambda(x)\end{pmatrix}\]
is a symmetric matrix. Each $A_\lambda(x)$ has two real eigenvalues $\mu^1_\lambda(x)$, $\mu^2_\lambda(x)$, which depend continuously on $(\lambda,x)\in I\times\overline{\Omega}$. We set for $\lambda\in I$

\begin{align*}
\alpha_\lambda&:=\inf_{x\in\overline{\Omega}}\{\mu^1_\lambda(x),\mu^2_\lambda(x)\}=\inf_{x\in\overline{\Omega}}\inf_{\|w\|=1}\langle A_\lambda(x)w,w\rangle_{\mathbb{R}^2},\\ \beta_\lambda&:=\sup_{x\in\overline{\Omega}}\{\mu^1_\lambda(x),\mu^2_\lambda(x)\}=\sup_{x\in\overline{\Omega}}\sup_{\|w\|=1}\langle A_\lambda(x)w,w\rangle_{\mathbb{R}^2},
\end{align*}
and note that these numbers can be easily obtained since $\mu^1_\lambda(x), \mu^2_\lambda(x)$ are just the zeros of quadratic polynomials. For example, for the systems \eqref{equII} we have

\begin{align}\label{alphabeta}
\begin{split}
\alpha_\lambda&=-\frac{a_\lambda+c_\lambda}{2}-\sqrt{\frac{1}{4}(a_\lambda-c_\lambda)^2+b^2_\lambda}\\
\beta_\lambda&=-\frac{a_\lambda+c_\lambda}{2}+\sqrt{\frac{1}{4}(a_\lambda-c_\lambda)^2+b^2_\lambda}.
\end{split}
\end{align}
Let us recall that we denote by $\{\lambda_k\}_{k\in\mathbb{N}}$ the sequence of Dirichlet eigenvalues of the domain $\Omega$ and that $0<\lambda_1\leq\lambda_2\leq\ldots$. Our main theorem of this section reads as follows.

\begin{theorem}\label{thm:comp}
Let $\Omega$ be a bounded domain having a smooth boundary and let us assume that (A1)-(A4) hold and that the linear equations \eqref{equlinI} have only the trivial solution for $\lambda=0,1$.

\begin{itemize}
	\item[(i)] If $\beta_0<\alpha_1$ and there exists $k\in\mathbb{N}$ such that 
	
	\begin{align}\label{condpos}
	\beta_0<\lambda_k<\alpha_1\quad\text{or}\quad \beta_0<-\lambda_k<\alpha_1,
	\end{align}
	then there is a bifurcation point for \eqref{equI}.
	\item[(ii)] If $\beta_1<\alpha_0$ and there exists $k\in\mathbb{N}$ such that 
	
	\begin{align}\label{condneg}
	\beta_1<\lambda_k<\alpha_0\quad\text{or}\quad \beta_1<-\lambda_k<\alpha_0,
	\end{align}
	then there is a bifurcation point for \eqref{equI}. 
\end{itemize}
\end{theorem}
\noindent
Let us point out again that no knowledge about solutions of the systems \eqref{equlinI} for $\lambda\in(0,1)$ is used in Theorem \ref{thm:comp}.

\begin{proof}
By definition of $\alpha_0,\alpha_1,\beta_0,\beta_1$, we have the inequalities

\begin{align}\label{estimate}
\alpha_0\, I_2\leq A_0(x)\leq\beta_0\,I_2,\quad \alpha_1\, I_2\leq A_1(x)\leq\beta_1\,I_2,\quad x\in\Omega, 
\end{align} 
where $I_2$ denotes the $2\times2$ identity matrix. Let us now consider the paths of matrices $\{B_\lambda\}_{\lambda\in I}$ and $\{C_\lambda\}_{\lambda\in I}$ given by

\[B_\lambda=(\beta_0+\lambda(\alpha_1-\beta_0)) I_2\quad\text{and}\quad C_\lambda=(\alpha_0+\lambda(\beta_1-\alpha_0)) I_2.\]
We obtain associated paths $M=\{M_\lambda\}_{\lambda\in I}$ and $N=\{N_\lambda\}_{\lambda\in I}$ in $\Phi_S(E)$ by setting

\begin{align*}
\langle M_\lambda z,\overline{z}\rangle_E&:=\langle Tz,\overline{z}\rangle_E+\int_\Omega{\langle B_\lambda z,\overline{z}\rangle\,dx},\quad
\langle N_\lambda z,\overline{z}\rangle_E:=\langle Tz,\overline{z}\rangle_E+\int_\Omega{\langle C_\lambda z,\overline{z}\rangle\,dx}
\end{align*}
and we note that by \eqref{Hessian} $L_\lambda-M_\lambda$ and $L_\lambda-N_\lambda$ are compact for all $\lambda\in I$. Since

\[\langle(L_\lambda-M_\lambda)z,\overline{z}\rangle_E=\int_\Omega{\langle(A_\lambda(x)-B_\lambda) z,\overline{z}\rangle\,dx},\quad z,\overline{z}\in E\]
and 

\[\langle(L_\lambda-N_\lambda)z,\overline{z}\rangle_E=\int_\Omega{\langle(A_\lambda(x)-C_\lambda) z,\overline{z}\rangle\,dx},\quad z,\overline{z}\in E,\]
we obtain from \eqref{estimate} and Proposition \ref{comparison} that

\begin{align}\label{sflestimate}
\sfl(M,I)\leq\sfl(L,I)\leq\sfl(N,I).
\end{align}
Because $L_0$ and $L_1$ are invertible by the assumption that \eqref{equlinI} has no non-trivial solutions for $\lambda=0$ and $\lambda=1$, the assertion follows from Theorem \ref{mainthmI} if we can prove that $\sfl(M,I)>0$ under the assumptions of (i), and  $\sfl(N,I)<0$ under the assumptions of (ii), respectively.\\
Let us first consider the path $M$. Since $M_0$ and $M_1$ are not necessarily invertible, we have by definition $\sfl(M,I)=\sfl(M^\delta,I)$, where $M^\delta:=\{M_\lambda+\delta\,I_E\}_{\lambda\in I}$ for an arbitrarily small $\delta>0$. From the results in Section \ref{sect:index}, we know that there is a decomposition of $E$ into two-dimensional subspaces $E_k$, $k\in\mathbb{N}$, such that the operator $T$ is reduced by this decomposition. Clearly, the $E_k$ reduce $M^\delta$ too, and moreover it is readily seen that

\begin{align*}
M^\delta_\lambda\mid_{E_k}=\begin{pmatrix}
-1+\delta&0\\
0&1+\delta
\end{pmatrix}
+\frac{\beta_0+\lambda(\alpha_1-\beta_0)}{\lambda_k}I_2,\quad\lambda\in I.
\end{align*}
By \eqref{sflformula} we know that 

\begin{align}\label{sflsgn}
\sfl(M,I)=\frac{1}{2}\sum^\infty_{k=1}{\sgn(M^\delta_1\mid_{E_k})}-\frac{1}{2}\sum^\infty_{k=1}{\sgn(M^\delta_0\mid_{E_k})}.
\end{align}
Let us now consider at first 

\[M^\delta_1\mid_{E_k}=\begin{pmatrix}
-1+\delta+\frac{\alpha_1}{\lambda_k}&0\\
0&1+\delta+\frac{\alpha_1}{\lambda_k}
\end{pmatrix}.\]
If $\alpha_1\geq 0$, then $1+\delta+\frac{\alpha_1}{\lambda_k}>0$ for all $k\in\mathbb{N}$ and consequently $\sgn( M^\delta_1\mid_{E_k})$ is either $0$ or $2$. The latter case happens if and only if $-1+\delta+\frac{\alpha_1}{\lambda_k}>0$ and, since $\delta>0$ is arbitrarily small, this is equivalent to $-1+\frac{\alpha_1}{\lambda_k}\geq 0$ and so $\alpha_1\geq \lambda_k$. If, on the other hand, $\alpha_1<0$, then $-1+\delta+\frac{\alpha_1}{\lambda_k}<0$ for all $k\in\mathbb{N}$ and so $\sgn(M^\delta_1\mid_{E_k})$ is either $0$ or $-2$. Here the latter case happens if $1+\delta+\frac{\alpha_1}{\lambda_k}<0$ which means that $\alpha_1<-\lambda_k$. In summary, we obtain

\[\frac{1}{2}\sum^\infty_{k=1}{\sgn(M^\delta_1\mid_{E_k})}=\begin{cases}
\#\{k\in\mathbb{N}:\,\alpha_1\geq\lambda_k\},\,\, &\text{if}\,\,\alpha_1\geq 0\\
-\#\{k\in\mathbb{N}:\,\alpha_1<-\lambda_k\},\,\, &\text{if}\,\,\alpha_1<0,
\end{cases}\]
and by the very same argument we also get that

\[\frac{1}{2}\sum^\infty_{k=1}{\sgn(M^\delta_0\mid_{E_k})}=\begin{cases}
\#\{k\in\mathbb{N}:\,\beta_0\geq\lambda_k\},\,\, &\text{if}\,\,\beta_0\geq 0\\
-\#\{k\in\mathbb{N}:\,\beta_0<-\lambda_k\},\,\, &\text{if}\,\,\beta_0<0.
\end{cases}\]
Consequently, it follows from \eqref{sflsgn} that

\begin{align}\label{sflM}
\sfl(M,I)=\begin{cases}
\#\{k\in\mathbb{N}:\,\alpha_1\geq\lambda_k\}-\#\{k\in\mathbb{N}:\,\beta_0\geq\lambda_k\},\,\, &\text{if}\,\, \alpha_1,\beta_0\geq0,\\
-\#\{k\in\mathbb{N}:\,\alpha_1<-\lambda_k\}-\#\{k\in\mathbb{N}:\,\beta_0\geq\lambda_k\},\,\, &\text{if}\,\, \alpha_1<0,\beta_0\geq 0,\\
\#\{k\in\mathbb{N}:\,\alpha_1\geq\lambda_k\}+\#\{k\in\mathbb{N}:\,\beta_0<-\lambda_k\},\,\, &\text{if}\,\, \alpha_1\geq 0,\beta_0<0,\\
-\#\{k\in\mathbb{N}:\,\alpha_1<-\lambda_k\}+\#\{k\in\mathbb{N}:\,\beta_0<-\lambda_k\},\,\, &\text{if}\,\, \alpha_1,\beta_0<0,
\end{cases}
\end{align}
which is positive if $\beta_0<\lambda_k<\alpha_1$ or $\beta_0<-\lambda_k<\alpha_1$ for some $k\in\mathbb{N}$. This finishes the proof of the first part of Theorem \ref{thm:comp}.\\ 
For the second part we need to show that $\sfl(N,I)=\sfl(N^\delta,I)<0$, where $N^\delta=\{N_\lambda+\delta\,I_E\}_{\lambda\in I}$ for an arbitrarily small $\delta>0$. We leave it to the reader to check that a similar argument as above shows that

\begin{align}\label{sflN}
\sfl(N,I)=\begin{cases}
\#\{k\in\mathbb{N}:\,\beta_1\geq\lambda_k\}-\#\{k\in\mathbb{N}:\,\alpha_0\geq\lambda_k\},\,\, &\text{if}\,\, \alpha_0,\beta_1\geq0,\\
-\#\{k\in\mathbb{N}:\,\beta_1<-\lambda_k\}-\#\{k\in\mathbb{N}:\,\alpha_0\geq\lambda_k\},\,\, &\text{if}\,\, \beta_1<0,\alpha_0\geq 0,\\
\#\{k\in\mathbb{N}:\,\beta_1\geq\lambda_k\}+\#\{k\in\mathbb{N}:\,\alpha_0<-\lambda_k\},\,\, &\text{if}\,\, \beta_1\geq 0,\alpha_0<0,\\
-\#\{k\in\mathbb{N}:\,\beta_1<-\lambda_k\}+\#\{k\in\mathbb{N}:\,\alpha_0<-\lambda_k\},\,\, &\text{if}\,\, \beta_1,\alpha_0<0,
\end{cases}
\end{align}
which is negative if $\beta_1<\lambda_k<\alpha_0$ or $\beta_1<-\lambda_k<\alpha_0$ for some $k\in\mathbb{N}$.
\end{proof}
\noindent
As an example of Theorem \ref{thm:comp}, let us consider once again the systems \eqref{equSzulkin}, where the matrix $A$ does not depend on $x\in\Omega$ and is linear in $\lambda$. Then we obtain from \eqref{alphabeta}

\begin{align}
\begin{split}
\alpha_0&=\beta_0=0\\
\alpha_1&=-\frac{a+c}{2}-\sqrt{\frac{1}{4}(a-c)^2+b^2}\\
\beta_1&=-\frac{a+c}{2}+\sqrt{\frac{1}{4}(a-c)^2+b^2}
\end{split}
\end{align} 
and see that there is a bifurcation point for \eqref{equSzulkin} in $(0,1)$ if

\[-\frac{a+c}{2}-\sqrt{\frac{1}{4}(a-c)^2+b^2}>\lambda_1\quad \text{or}\quad -\frac{a+c}{2}+\sqrt{\frac{1}{4}(a-c)^2+b^2}<-\lambda_1.\]


\section{The case $N=1$}
In this section we consider the special case that $N=1$, i.e. $\Omega$ is a bounded interval in $\mathbb{R}$. For the sake of simplicity, we restrict to the case $\Omega=(0,\pi)$ and so the systems \eqref{equI} are of the form

\begin{equation}\label{equIODE}
\left\{
\begin{aligned}
-u''&= b_\lambda(x)u+c_\lambda(x)v+G'_v(\lambda,x,u,v)&& \,\text{in}\,\,(0,\pi),\\
-v''&= a_\lambda(x)u+b_\lambda(x)v+G'_u(\lambda,x,u,v)&& \,\text{in}\,\,(0,\pi),\\
u(0)&=v(0)=u(\pi)=v(\pi)=0.
\end{aligned}
\right.
\end{equation}
We want to show that our previous results can be used to obtain an estimate on the number of bifurcation points for \eqref{equIODE}. Let us note for later reference the corresponding linearised equations, which are

\begin{equation}\label{equIlinODE}
\left\{
\begin{aligned}
-u''&= b_\lambda(x)u+c_\lambda(x)v&& \,\text{in}\,\,(0,\pi),\\
-v''&= a_\lambda(x)u+b_\lambda(x)v&& \,\text{in}\,\,(0,\pi),\\
u(0)&=v(0)=u(\pi)=v(\pi)=0.
\end{aligned}
\right.
\end{equation}
Before we can state our main result of this section, we need to make another digression about a property of the spectral flow. Let us assume that $L=\{L_\lambda\}_{\lambda\in I}$ is a path of selfadjoint Fredholm operators such that $L_\lambda$ is non-invertible only at the finite number of instants $0<\lambda_1\leq\ldots\leq\lambda_m<1$. Then, by the Concatenation Property of the spectral flow, there is $\varepsilon>0$ such that

\[\sfl(L,I)=\sum^m_{i=1}{\sfl(L,[\lambda_i-\varepsilon,\lambda_i+\varepsilon])}.\] 
From the construction of the spectral flow in \cite{SFLPejsachowicz}, it is intuitively clear (however, not trivial to prove rigorously (cf. \cite[Lemma 4.5]{BifJac})) that

\begin{align}\label{kernelestimate}
|\sfl(L,[\lambda_i-\varepsilon,\lambda_i+\varepsilon])|\leq\dim\ker L_{\lambda_i}.
\end{align}
Let now $f:I\times H\rightarrow\mathbb{R}$ be a family of functionals as in Section \ref{sect:setting} such that $L_\lambda$ is the Riesz representation of $D^2_0f_\lambda$ for $\lambda\in I$ as in \eqref{Riesz}. By the implicit function theorem, if $\lambda^\ast$ is a bifurcation point for $f$, then $\lambda^\ast=\lambda_i$ for some $1\leq i\leq m$. Moreover, $\lambda_i$ is a bifurcation point if $\sfl(L,[\lambda_i-\varepsilon,\lambda_i+\varepsilon])\neq 0$ by Theorem \ref{mainthmI}. From these facts, the following result is readily seen (cf. \cite[Thm 2.1 (ii)]{BifJac}).

\begin{lemma}\label{bifest}
Let $f:I\times H\rightarrow\mathbb{R}$ and $L=\{L_\lambda\}_{\lambda\in I}$ be as in Section \ref{sect:setting}. We assume that $L$ is admissible and $L_\lambda$ is non-invertible for only a finite number of $\lambda\in(0,1)$. Then the number of bifurcation points for $f$ is bounded below by

\[\frac{|\sfl(L,I)|}{\max_{\lambda\in(0,1)}\dim\ker L_\lambda}.\]
\end{lemma} 
\noindent
We now introduce a natural number $\Gamma(\alpha,\beta)$ for any pair of real numbers $\alpha>\beta$ by

\begin{align*}
\Gamma(\alpha,\beta)=\begin{cases}
\#\{k\in\mathbb{N}:\,\alpha\geq k^2\geq\beta\},\,&\text{if}\,\alpha,\beta\geq0,\\
\#\{k\in\mathbb{N}:\,\alpha\geq k^2\}+\#\{k\in\mathbb{N}:\,\beta<-k^2\}\,&\text{if}\,\alpha\geq0,\,\beta<0,\\
\#\{k\in\mathbb{N}:\,\beta<- k^2<\alpha\},\,&\text{if}\,\alpha,\beta<0,
\end{cases}
\end{align*}
which we need to state our main result of this section.

\begin{theorem}\label{thm:ODE} 
Let $\Omega=(0,\pi)\subset\mathbb{R}$ and let us assume that (A1),(A2) and (A3) hold. We suppose that there are only finitely many $\lambda\in(0,1)$ for which the linear equation \eqref{equIlinODE} has a non-trivial solution and, moreover, we assume that there is only the trivial solution for $\lambda=0,1$.

\begin{itemize}
	\item[(i)] If $\alpha_1>\beta_0$,  then there are at least $\frac{1}{2}\Gamma(\alpha_1,\beta_0)$ bifurcation points for \eqref{equIODE}.
	\item[(ii)] If $\alpha_0>\beta_1$,  then there are at least $\frac{1}{2}\Gamma(\alpha_0,\beta_1)$	bifurcation points for \eqref{equIODE}. 
\end{itemize}
\end{theorem} 

\begin{proof}
In the proof of Theorem \ref{thm:comp}, we constructed paths $M$ and $N$ such that $\sfl(M,I)\leq\sfl(L,I)$ and $\sfl(L,I)\leq\sfl(N,I)$, respectively. Note that the Dirichlet eigenvalues of the domain $\Omega=(0,\pi)$ are $\lambda_k=k^2$, $k\in\mathbb{N}$. If now $\alpha_1>\beta_0$, we obtain from \eqref{sflM} that $\sfl(L,I)\geq\sfl(M,I)=\Gamma(\alpha_1,\beta_0)$. If, however, $\alpha_0>\beta_1$, we get from \eqref{sflN} that $\sfl(L,I)\leq\sfl(N,I)=-\Gamma(\alpha_0,\beta_1)$ and so $|\Gamma(\alpha_0,\beta_1)|\leq|\sfl(L,I)|$.\\ 
Finally, the result follows from Lemma \ref{bifest} if we note that $\dim\ker L_\lambda\leq 2$ as the kernel of $L_\lambda$ consists of solutions of the 2-dimensional system of linear ordinary equations \eqref{equIlinODE}.
\end{proof}
\noindent
Finally, let us consider the equations \eqref{equSzulkin} on $\Omega=(0,\pi)$. It follows from the spectral theory of compact operators that the corresponding linearised equations \eqref{equIlinODE} can only have a non-trivial solution for a finite number of values of the parameter $\lambda$. Hence if we assume that there is only the trivial solution for $\lambda=1$, then we obtain that there at least

\[\frac{1}{2}\max\{k\in\mathbb{N}:-\frac{a+c}{2}-\sqrt{\frac{1}{4}(a-c)^2+b^2}\geq k^2\}\]
or
\[\frac{1}{2}\max\{k\in\mathbb{N}:-\frac{a+c}{2}+\sqrt{\frac{1}{4}(a-c)^2+b^2}\geq -k^2\}\]
distinct bifurcation points for the nonlinear equations \eqref{equSzulkin}, where only one of these numbers can be non-zero. In particular, we can easily construct systems having an arbitrarily high number of bifurcation points in $(0,1)$.


\appendix

\section{Families of Compact Operators and Strong Convergence}\label{CompactStrongConv}
In this appendix we prove a well known assertion, which however we could not find in the literature. In what follows, we let $X,Y$ be Banach spaces and we denote by $\mathcal{L}(X,Y)$ the Banach space of all bounded linear operators with respect to the operator norm. 

\begin{lemma}\label{app-lemma}
Let $\Lambda$ be a compact metric space and $K=\{K_\lambda\}_{\lambda\in\Lambda}$ a continuous family in $\mathcal{L}(X,Y)$ such that $K_\lambda:=K(\lambda,\cdot):X\rightarrow Y$ is compact for all $\lambda\in\Lambda$. Then $K(\Lambda\times B)$ is relatively compact for every bounded subset $B\subset X$.  
\end{lemma}

\begin{proof}
Let $\{(\lambda_n,x_n)\}_{n\in\mathbb{N}}$ be a sequence in $\Lambda\times B$ and $c>0$ such that $\|x_n\|\leq c$, $n\in\mathbb{N}$. Since $\Lambda$ is sequentially compact, we can find a subsequence $\{\lambda_{n_k}\}_{k\in\mathbb{N}}\subset \{\lambda_n\}_{n\in\mathbb{N}}$ converging to some $\lambda^\ast\in \Lambda$. Moreover, by the compactness of $K_{\lambda^\ast}$ we can thin out $\{(\lambda_{n_k},x_{n_k})\}_{k\in\mathbb{N}}$ to obtain a subsequence $\{(\lambda_{n_l},x_{n_l})\}_{l\in\mathbb{N}}\subset\{(\lambda_n,x_n)\}_{n\in\mathbb{N}}$ such that $\lambda_{n_l}\rightarrow \lambda^\ast$ and $K(\lambda^\ast,x_{n_l})$ converges to some $y\in Y$. Then

\begin{align*}
\|K(\lambda_{n_l},x_{n_l})-y\|&\leq \|K(\lambda_{n_l},x_{n_l})-K(\lambda^\ast,x_{n_l})\|+\|K(\lambda^\ast,x_{n_l})-y\|\\
&\leq \|K_{\lambda_{n_l}}-K_{\lambda^\ast}\|\|x_{n_l}\|+\|K(\lambda^\ast,x_{n_l})-y\|\\
&\leq c\,\|K_{\lambda_{n_l}}-K_{\lambda^\ast}\|+\|K(\lambda^\ast,x_{n_l})-y\|\rightarrow 0,
\end{align*} 
where the first term converges to zero because of the continuity of the family $K$ with respect to the norm topology.
\end{proof}
\noindent
Let us recall that a sequence $\{S_k\}_{k\in\mathbb{N}}$ of operators in $\mathcal{L}(X,Y)$ is called strongly convergent to $S\in\mathcal{L}(X,Y)$ if $\|S_ku-Su\|\rightarrow0$ as $k\rightarrow\infty$ for every $u\in X$. Clearly, every sequence that is convergent in $\mathcal{L}(X,Y)$ (i.e. with respect to the norm topology) is also strongly convergent.

\begin{cor}\label{appendix}
Let $K=\{K_\lambda\}_{\lambda\in\Lambda}$ be a family of compact operators as in Lemma \ref{app-lemma} and let $\{S_k\}_{k\in\mathbb{N}}$ be a sequence in $\mathcal{L}(Y)$ such that $S_k$ converges to $S\in\mathcal{L}(Y)$ strongly. Then

\[\sup_{\lambda\in\Lambda}\|S_kK_\lambda-SK_\lambda\|\rightarrow 0,\quad k\rightarrow\infty,\]
i.e., $S_kK_\lambda$ converges to $SK_\lambda$ in $\mathcal{L}(X,Y)$ uniformly in $\lambda\in\Lambda$.
\end{cor}

\begin{proof}
We note at first that by the Uniform Boundedness Principle there is a constant $C>0$ such that $\|S_k\|<C$ for all $k\in\mathbb{N}$, and we assume in addition that $C$ is greater than $\|S\|$. Let $B_r(u)$ denote the ball of radius $r>0$ around $u$. By Lemma \ref{app-lemma}, the set $K(\Lambda\times B_1(0))$ is relatively compact in $Y$ and hence for every $\varepsilon>0$ there exist $y_1,\ldots,y_N\in Y$ such that 

\[K(\Lambda\times B_1(0))\subset \bigcup^N_{i=1}{B_{\frac{\varepsilon}{3C}}(y_i)}.\]
Let now $x\in B_1(0)$ and $\lambda\in\Lambda$ be arbitrary. We choose $y_i\in Y$ such that $K_\lambda x\in B_{\frac{\varepsilon}{3C}}(y_i)$ and obtain

\begin{align*}
\|S_kK_\lambda x-SK_\lambda x\|\leq\|S_k\|\|K_\lambda x-y_i\|+\|S_ky_i-Sy_i\|+\|S\|\|y_i-K_\lambda x\|\leq \frac{2\varepsilon}{3}+\|S_ky_i-Sy_i\|.
\end{align*}  
Hence if we choose $k_0$ so large that $\max_{i=1,\ldots,N}\{\|S_ky_i-Sy_i\|\}<\frac{\varepsilon}{3}$ for all $k\geq k_0$, then we obtain

\[\sup_{\lambda\in\Lambda}\sup_{\|x\|=1}\|S_kK_\lambda x-SK_\lambda x\|<\varepsilon,\quad\text{for}\,\, k\geq k_0.\]
\end{proof}
\noindent
Finally, let us consider a separable Hilbert space $H$ with scalar product $\langle\cdot,\cdot\rangle$ and let us assume that $\{e_k\}_{k\in\mathbb{N}}$ is an orthonormal basis of $H$. We obtain a sequence of orthogonal projections by setting 

\[P_nu=\sum^n_{k=1}{\langle u,e_k\rangle e_k},\quad n\in\mathbb{N},\]
which converges strongly to the identity operator $I_H$. As a consequence, if $K=\{K_\lambda\}_{\lambda\in\Lambda}$ is a family of compact operators in $\mathcal{L}(H)$ parametrised by the compact metric space $\Lambda$, then

\[\sup_{\lambda\in\Lambda}\|P_nK_\lambda-K_\lambda\|\rightarrow 0,\quad n\rightarrow\infty.\]
 

\vspace{1cm}
Nils Waterstraat\\
School of Mathematics,\\
Statistics \& Actuarial Science\\
University of Kent\\
Canterbury\\
Kent CT2 7NF\\
UNITED KINGDOM\\
E-mail: n.waterstraat@kent.ac.uk

\end{document}